\documentclass[reqno]{amsart}

\makeatletter
\renewcommand\part{%
	\if@noskipsec \leavevmode \fi
	\par
	\addvspace{4ex}%
	\@afterindentfalse
	\secdef\@part\@spart}

\def\@part[#1]#2{%
	\ifnum \c@secnumdepth >\m@ne
	\refstepcounter{part}%
	\addcontentsline{toc}{part}{\thepart\hspace{1em}#1}%
	\else
	\addcontentsline{toc}{part}{#1}%
	\fi
	{\parindent \z@ \raggedright
		\interlinepenalty \@M
		\normalfontsymmetric decreasing rearrangement
		\ifnum \c@secnumdepth >\m@ne
		\Large\bfseries \partname\nobreakspace\thepart
		\par\nobreak
		\fi
		\huge \bfseries #2%
		\par}%
	\nobreak
	\vskip 3ex
	\@afterheading}
\def\@spart#1{%
	{\parindent \z@ \raggedright
		\interlinepenalty \@M
		\normalfont
		\huge \bfseries #1\par}%
	\nobreak
	\vskip 3ex
	\@afterheading}
\makeatother
\usepackage{color}
\usepackage[dvipsnames]{xcolor}
\usepackage{ifpdf}
\ifpdf 
    \usepackage[pdftex]{graphicx}   
    \usepackage[pdftex,     
            plainpages=false,   
            breaklinks=true,    
            colorlinks=true,
            linkcolor=red,
            citecolor=green,
            pdftitle=My Document
            pdfauthor=My Good Self
           ]{hyperref} 
\else 
    \usepackage{graphicx}       
\fi 

\usepackage{subfig}

\usepackage{glossaries}
\usepackage{glossary-mcols}

\usepackage{aurical}
\usepackage{amsfonts,amsmath}	
\usepackage{amssymb}
\usepackage{verbatim}
\usepackage{amsopn}
\usepackage[english]{babel}
\usepackage{amsthm}
\usepackage{enumerate}
\usepackage{mathrsfs}	
\usepackage{enumitem}
\usepackage{mathtools}
\usepackage{esint}
\usepackage{bbm}
\usepackage{caption}
\usepackage{marginnote}
\usepackage[marginparwidth=2cm]{geometry}
\usepackage{csquotes}
\usepackage{cleveref}
\crefname{enumi}{part}{parts}
\date{\today}

\theoremstyle{definition} \newtheorem{definition}{Definition}[section]
\theoremstyle{definition} 
\theoremstyle{plain} \newtheorem{lemma}[definition]{Lemma}
\theoremstyle{plain} \newtheorem{proposition}[definition]{Proposition}
\theoremstyle{plain} \newtheorem{theorem}[definition]{Theorem}
\theoremstyle{plain} \newtheorem{corollary}[definition]{Corollary}
\theoremstyle{definition} 
\theoremstyle{plain} 
\theoremstyle{definition} 
\theoremstyle{definition}

\DeclareMathOperator{\dive}{div}

\newcommand{\norm}[1]{\lVert #1\rVert}

\newcommand{\R}{\mathbb{R}}

\newcommand{\loc}{\text{\rm loc}}

\numberwithin{equation}{section} 

\theoremstyle{plain} \newtheorem*{theorem*}{Theorem}
\theoremstyle{plain} 
\theoremstyle{plain} \newtheorem*{mthm*}{Main Theorem}
\theoremstyle{plain} \newtheorem*{conjecture*}{Conjecture}
\theoremstyle{plain} 
\theoremstyle{plain} \newtheorem*{problem*}{Problem}

\usepackage{biblatex} 
\title{Existence and Blow-Up for Non-linear Fokker-Planck with Controlled Drift}

\author{Giacomo Maria Leccese}
\email{giacomo.leccese94@gmail.com}

\begin{document}

\begin{abstract}
We extend the global existence results for critical parameters of \cite{bianchini2024existence} to the multidimensional problem
\begin{equation*}
    \partial_t u + \text{div } (b(t,x) u^{1+k}) = \Delta u
\end{equation*}
for $b:(0,\infty)\times\mathbb{R}^d\to\mathbb{R}^d$ non-autonomous fields, in the case
\begin{equation*}
   b\in L^\infty_{\mathrm{loc}}
  \bigl((0,\infty),L^{p,\infty}(\mathbb{R}^d)\bigr),
  \qquad p>d\ge 2.
\end{equation*}
For the critical case, we study the difference between two ordered solutions whose conserved
small mass absorbs the drift. Finitely many increments then reach an arbitrary solution.
\end{abstract}

\maketitle

\section{Introduction}

We study the Cauchy problem for the nonlinear Fokker--Planck equation
\begin{equation}\label{eq:pde}
\begin{cases}
\partial_tu+\dive\bigl(b(t,x)u^{1+k}\bigr)=\Delta u, &(t,x)\in(0,\infty)\times\R^d,\\
u(0,x)=u_0(x), & x\in\R^d,
\end{cases}
\end{equation}
with $k>0$ and a non-negative datum $u_0\in L^1(\R^d)\cap L^\infty(\R^d)$, under the sole
hypothesis
\begin{equation}\label{eq:assumptions}
b\in L^\infty_{\loc}\bigl((0,\infty),L^{p,\infty}(\R^d)\bigr),\qquad p>d\ge 2,
\end{equation}
on the prescribed field: no bound on $\dive b$, no sign condition and no boundedness of
$b$ are assumed.

Two mechanisms compete in \eqref{eq:pde}. The diffusion spreads the density, while the superlinear drift $ \text{div }(bu^{1+k})$ concentrates it.
Whether a bounded solution survives for all times is decided by which of the two terms prevails at small scales.
The balance is measured by the mass-preserving scaling
$u_\lambda(t,x)=\lambda^du(\lambda^2t,\lambda x)$, which turns a solution of \eqref{eq:pde} into a solution with field
$b_\lambda(t,x)=\lambda^{1-dk}b(\lambda^2t,\lambda x)$.
Since $\norm{b_\lambda(t)}_{p,\infty}=\lambda^{1-dk-d/p}\norm{b(\lambda^2t)}_{p,\infty}$,
the hypothesis \eqref{eq:assumptions} is scale-invariant exactly at
\begin{equation}\label{eq:critical-expo}
k=k_c:=\frac1d-\frac1p>0 ,
\end{equation}
and \eqref{eq:critical-expo} is the threshold at which global existence is decided.
We prove it for every $0<k\le k_c$, and we exhibit coefficients and data with finite-time blow-up above it.
At $k=k_c$ no smallness is assumed, neither on the mass of $u_0$ nor on the norm of $b$.
Being in divergence form, the equation conserves the mass of non-negative bounded solutions,
$\norm{u(t)}_1=\norm{u_0}_1$. This is what makes the critical exponent delicate, since the only conserved quantity available is prescribed by the datum and cannot be assumed small.

Equations of the form \eqref{eq:pde} describe the kinetic evolution of particle systems obeying an exclusion or an inclusion principle, following the models of Kaniadakis and
Quarati \cite{kaniadakis1993kinetic,kaniadakis1994classical}. 
For the fermionic and bosonic equations, their equilibria and their finite-time condensation we refer to
\cite{carrillo20081d,toscani2012finite,toscani2025supercritical}, and to \cite{alikakos1989blow,escobedo1991large,kaplan1963growth,kieffer2025robust} for related
blow-up and large-time questions in advection--diffusion.
In one space dimension, \cite{guidolin2022global} obtained global existence in the subcritical regime for bounded Lipschitz drifts,
and \cite{bianchini2024existence} gave a complete classification under weak-Lebesgue hypotheses on $b$ or on $b_x$.
For $b\in L^{p,\infty}(\R)$ the threshold found there is $k=1-\frac1p$, which is \eqref{eq:critical-expo} with $d=1$.
The present paper and \cite{leccese2026existence} extend the two halves of that classification to
$d\ge2$. \cite{leccese2026existence} assumes a bound on $(\dive b)_-$, with the different threshold $\frac2d-\frac1p$,
whereas \eqref{eq:assumptions} constrains the field itself.
The two hypotheses are independent. A constant field has vanishing divergence and lies in no $L^{p,\infty}(\R^d)$,
while $b=\mathbf 1_{|x|<1}e_1$ belongs to every $L^{p,\infty}(\R^d)$ and its divergence is not a function.

With the present paper the classification started in \cite{bianchini2024existence} for $d=1$ and continued in \cite{leccese2026existence} under a bound on $(\dive b)_-$ is complete.
In each case global existence holds up to the critical exponent included, and blow-up occurs above it.
The critical case here is not obtained by the methods of the two previous papers.
A direct application of the one-dimensional argument cannot obtain the multidimensional critical case.
Indeed, the truncated entropy method used in \cite{bianchini2024existence} is based on the embedding $H^1(\R)\hookrightarrow L^\infty(\R)$,
which fails in higher dimensions. At the same time it is not clear to the author how to extend the argument of \cite{leccese2026existence},
based on the study of the mass concentration function $m(t,s)=\int_0^su^*(t,r)\,\mathrm{d}r$,
 to the case \eqref{eq:assumptions} for $p<\infty$. That method consists of a pointwise comparison in $s$,
and requires the drift contribution on $\{u>u^*(t,s)\}$ to be bounded by a function, uniformly over the sets of measure $s$. 
Such a requirement seems hard to combine with an integral condition on $b$.
In this paper we use a different comparison method. Given $u\ge v$ solutions of \eqref{eq:pde}, 
their increment $w=u-v$ satisfies $\partial_t w+\text{div }(b\,H(v,w))=\Delta w$ with $H(v,w)=(v+w)^{1+k}-v^{1+k}$. 
The subadditivity of $s\mapsto s^{k}$ splits the drift into a critical term in $w$ alone and a term linear in $w$ with coefficient $\norm{v(t)}_\infty^{k}$. 
The mass $\norm{w(t)}_1$ is conserved exactly, and when it is below an explicit threshold the Gagliardo--Nirenberg inequality absorbs the critical term. 
A bound on $w$ then transfers the bound on $v$ to $u$, and an arbitrary mass is reached in finitely many steps.
The argument uses neither the embedding $H^1\hookrightarrow L^\infty$ nor rearrangements,
but only the conservation of the mass of the increment and the Gagliardo--Nirenberg inequality.
\begin{theorem}\label{thm:global-main}
Assume conditions \eqref{eq:assumptions}.  If $0<k\le k_c$,
then the solution of \eqref{eq:pde} is globally defined.
More precisely, for every $T>0$ there is a finite constant $C$, 
depending only on $d,p,k,\norm{u_0}_1,\norm{u_0}_\infty$ and $\sup_{t\in[0,T]}\norm{ {b(t)}}_{L^{p,\infty}}$, such that
\begin{equation*}
  \sup_{0\leq t\leq T}\norm{u(t)}_\infty\leq C.
\end{equation*}
\end{theorem}

In the supercritical domain, the energy method of \cite{ bianchini2024existence,toscani2012finite} can be generalised for the multidimensional problem as follows.

\begin{theorem}\label{thm:blowup-main}
Assume \eqref{eq:assumptions} and $k>k_c$. There exist an autonomous, compactly supported
radial vector field
\[
b\in L^{p}(\R^d,\R^d)\subset L^{p,\infty}(\R^d,\R^d),
\]
bounded if $p=\infty$, and a non-negative $u_0\in L^1(\R^d)\cap L^\infty(\R^d)$ with
$\int_{\R^d}|x|^2u_0(x)\,dx<\infty$, such that the corresponding solution of \eqref{eq:pde}
blows up in the $L^\infty$-norm in finite time.
\end{theorem}

The construction is the one of \cite[Section 5]{leccese2026existence}, and only the
admissible range of the exponent changes. One takes a radial field with
$x\cdot b=-|x|\,|b|$ and $|b(x)|\simeq|x|^{\alpha}$ near the origin; the second moment
$E(t)=\int|x|^2u\,dx$ then obeys
\[
E'\le 2d\norm{u_0}_1-C\,E^{-\sigma/2},\qquad \sigma=dk-1-\alpha,
\]
so that $E$ vanishes in finite time whenever $E(u_0)$ is small enough, which is
incompatible with a bounded density of positive mass. The first condition on $\alpha$ is
$dk>1+\alpha$, which makes $\sigma>0$; the second is the membership of the field in the
space prescribed by the hypothesis. Here $|x|^{\alpha}\in L^p$ near the origin if and only
if $\alpha>-d/p$, that is $1+\alpha>1-\frac dp$, whereas in \cite{leccese2026existence}
the relevant quantity is $\dive b\simeq|x|^{\alpha-1}$ and the condition reads
$1+\alpha>2-\frac dp$. An exponent with $1-\frac dp<1+\alpha<dk$ exists precisely when
$dk>1-\frac dp$, that is $k>k_c$; for $p=\infty$ every $\alpha\in(0,dk-1)$ is admissible
and the field is bounded. We refer to \cite[Proposition 5.1]{leccese2026existence} for the
details.

For the long-time behaviour of the solutions, we consider the assumptions
\begin{equation}\label{eq:assumptions2}
\begin{split}
& k>0, \ p>d\ge 2,\\
&b\in L^\infty
  \bigl((0,\infty);L^{p,\infty}(\R^d)\bigr).
\end{split}
\end{equation}

The scenario for the long-time behaviour is described in the following theorem.
\begin{theorem}
    \label{Theo:long_time_1}
    Assume condition \eqref{eq:assumptions2}. The following holds.
    \begin{enumerate}
        \item \label{Point_1:long_time_1} If $k <k_c$, there are stationary solutions.

        \item \label{Point_2:long_time_1} If $k =k_c$, then every solution $u$ decays like $\|u(t)\|_\infty\leq \mathcal O(t^{-\frac{d}{2}}) $.
        \item \label{Point_3:long_time_1} If $k>k_c$, and the solution $u$ is uniformly bounded in time and space, i.e. $u \in L^\infty_{t,x}$, then it decays like $\|u(t)\|_\infty\leq \mathcal O(t^{-\frac{d}{2}}) $.
    \end{enumerate}
\end{theorem}
The paper is organised as follows. In \Cref{sec:preliminaries}, we introduce the notation, recall Lorentz spaces, and prove some preliminary lemmas.
In \Cref{sec:local}, we state the local well-posedness of the solution.
\Cref{sec:globex} is dedicated to the proof of global existence: first the subcritical case and the passage from $L^2$ to $L^\infty$, then the $L^2$ estimate for the increment of two ordered solutions in the critical case,
and the chain of solutions which reaches an arbitrary mass.
In \Cref{sec:longbehav}, we study the decay behaviour of the solution.

\section{Notation and Preliminaries}\label{sec:preliminaries}

We will use the following notation.
\begin{itemize}
	\item The letters $C_Y, C_{HK},C_I,\dots$ denote constants that will be clearly identifiable from the context, while $C, C(k),\dots$ is a universal constant that might change from line to line.
	\item $B_T=\sup_{t \in [0,T]}\norm{ {b(t)}}_{L^{p,\infty}}$.
	\item $p'=2p/(p-2)$, so that $\frac{1}{2}=\frac 1 p +\frac 1 {p'}$.
	\item For every function $f$ in two variables $t,x$, we denote by $f(t)$ the function $x\mapsto f(t,x)$. 
\end{itemize}
\subsection{Lorentz spaces}
\label{sec:lorentz}

We briefly recall the definition and some well known results about the Lorentz space. For any $f:\R^d\to \R$, let $f^*:[0,\infty)\to \R$ be the symmetric decreasing rearrangement defined by
\begin{equation*}
	f^{\ast}(z) = \inf \Big\{\alpha>0: \mathcal L^d\big(\big\{|f|>\alpha\big\}\big)|\leq z\Big\}.
\end{equation*}
Define 
\begin{equation}
	\label{Equa:Lorentz_2}
	\|f\|_{p,q} = \begin{cases}
		\left( \displaystyle \int_0^{\infty} \big[ z^{\frac{1}{p}} f^{*}(z) \big]^q \, \frac{dz}{z} \right)^{\frac{1}{q}} & q \in [1, \infty), p \in [1,\infty), \\
		\sup\limits_{z > 0} \, z^{\frac{1}{p}}  f^{*}(z)   & q = \infty, p \in [1,\infty).
	\end{cases}
\end{equation}
with the notation $\|f\|_{\infty,\infty}=\|f\|_{\infty}$.
\begin{proposition}
	$L^{p}$ is embedded in $L^{p,\infty}$. In particular for every function $f:\R^d \to \R$,
	$$\|f\|_{p,\infty}\le \|f\|_p$$
\end{proposition}
See \cite[Proposition 1.1.6]{grafakos2008classical}
\begin{theorem}[Hardy–Littlewood inequality]
	Given $f_1,f_2$ non-negative functions, it holds
	$$
	\int _{\mathbb {R} ^{n}}f_1(x)f_2(x)\,dx\leq \int _{\mathbb {R} ^{n}}f_1^{*}(z)f_2^{*}(z)\,dz.$$
\end{theorem}
See \cite{hardy1952inequalities}.
\begin{theorem}[H\"older’s inequality in Lorentz spaces]
	\label{Theo:holder_lorentz}
	Let $p\in[1,\infty)$, $ p_1\in [1,\infty), p_2 \in [1,\infty]$ such that 
	$$\frac{1}{p}=\frac{1}{p_1}+\frac{1}{p_2},$$
	then 
	$$
	\|f_1f_2\|_{p} \le  \|f_1\|_{p_1,p} \|f_2\|_{p_2,\infty}.
	$$
\end{theorem} 

\begin{proof}
	The case $p_2=\infty$ is trivial. For $p_2<\infty$, by Hardy–Littlewood inequality
	\begin{equation*}
		\begin{split}
			\|f_1f_2\|_p^p&\le\int(|f_1|^p)^{*}(|f_2|^p)^{*}\\
			&=\int(f_1^{*})^p(f_2^{*})^p\\
			&\le \sup z^\frac{p}{p_2}(f_2^{*})^p\int (z^\frac{1}{p_1}f^*)^p\frac{dz}{z}\\
			&=\|f_1\|_{p_1,p}^p \|f_2\|_{p_2,\infty}^p,
		\end{split}
	\end{equation*}
	where we used $\left(|f|^{p}\right)^{*}=(f^{*})^{p}$, that is easy to prove.
\end{proof}
\begin{theorem}[Young inequality in Lorentz spaces]
\label{Theo:young_lorentz}
Let $p,p_1\in(1,\infty)$, $p_2\in[1,\infty)$ such that 
	$$\frac{1}{p}+1=\frac{1}{p_1}+\frac{1}{p_2},$$
	then there is a constant $C_Y=C_Y(p,p_1,p_2)$ such that
	$$
	\|f_1\star f_2\|_{p,\infty} \le  C_Y\|f_1\|_{p_1,\infty} \|f_2\|_{p_2}.
	$$
\end{theorem} 
See, for instance, \cite[Theorem 1.4.25]{grafakos2008classical}.

\begin{lemma}[Interpolation of Lorentz space between Lebesgue spaces]\label{lem:interp}
Let $1\le r_1<r<r_2\le\infty$, let $\theta\in(0,1)$ be defined by
\begin{equation*}
\frac1r=\frac{\theta}{r_1}+\frac{1-\theta}{r_2},
\end{equation*}
and let $s\in[1,\infty]$. Then there is a constant $C_I=C_I(r,r_1,r_2,s)$ such that
\begin{equation*}
\norm{f}_{r,s}\le C_I\norm{f}_{r_1}^{\theta}\norm{f}_{r_2}^{1-\theta}
\qquad\text{for every }f\in L^{r_1}\cap L^{r_2}.
\end{equation*}
\end{lemma}
 
\begin{proof}
By Chebyshev's inequality, $f^*(z)\le z^{-1/r_i}\norm{f}_{r_i}$ for $i=1,2$, with the convention $z^{-1/\infty}=1$. Let $s<\infty$ and $\bar z>0$. Splitting the integral in \eqref{Equa:Lorentz_2} at $\bar z$ and using the bound with $i=2$ on $(0,\bar z)$ and with $i=1$ on $(\bar z,\infty)$,
\begin{equation*}
\norm{f}_{r,s}^s\le\frac{\norm{f}_{r_2}^s}{s\bigl(\frac1r-\frac1{r_2}\bigr)}\,
\bar z^{\,s(\frac1r-\frac1{r_2})}
+\frac{\norm{f}_{r_1}^s}{s\bigl(\frac1{r_1}-\frac1r\bigr)}\,
\bar z^{-s(\frac1{r_1}-\frac1r)}.
\end{equation*}
 Choosing
\begin{equation*}
\bar z=\Bigl(\frac{\norm{f}_{r_1}}{\norm{f}_{r_2}}\Bigr)^{(\frac1{r_1}-\frac1{r_2})^{-1}}
\end{equation*}
makes the two terms equal and gives the statement. The case $s=\infty$ follows analogously from the same two bounds, taking the supremum separately on $(0,\bar z)$ and on $(\bar z,\infty)$.
\end{proof}

\begin{proposition}[Gagliardo--Nirenberg interpolation inequality for Lorentz space,\cite{nirenberg1959elliptic}]
\label{Prop:gagliardo}Let $q,r\in[1,\infty]$, let $j$ and $m$ non-negative integers such that $j<m$. Then for $\theta \in [0,1]$ such that

\begin{equation*}
    \frac {1}{p}={\dfrac {j}{d}}+\theta \left({\dfrac {1}{r}}-{\dfrac {m}{d}}\right)+{\dfrac {1-\theta }{q}},\qquad {\dfrac {j}{m}}\leq \theta \leq 1
\end{equation*}
there exists a constant $C_G=C_G(q,r,j,m,d)$ such that for every function $f$
\begin{equation*}
    \|D^{j}f\|_{L^{p}(\mathbb {R} ^{d})}\leq C_G\|D^{m}f\|_{L^{r}(\mathbb {R} ^{d})}^{\theta }\|f\|_{L^{q}(\mathbb {R} ^{d})}^{1-\theta }
\end{equation*}
with two further conditions:
\begin{enumerate}
    \item if $j=0$  $q=+\infty$  and $rm<d$, then we need to assume that either $f$ tends to $0$ at infinity, or  $f\in L^{s}(\mathbb {R}^{d})$ for some finite value of $s$,
    \item if $r>1$ and  $m-j-{\frac {d}{r}}$ is a non-negative integer, then we need to assume that $\frac {j}{m}\leq \theta <1$.
\end{enumerate}
Moreover if $p\ne \infty$ and $\theta\in(\frac{j}{m},1)$, for every $\alpha\ge 1$, it holds
\begin{equation*}
    \|D^{j}f\|_{L^{p,\alpha}(\mathbb {R} ^{d})}\leq C_G\|D^{m}f\|_{L^{r}(\mathbb {R} ^{d})}^{\theta }\|f\|_{L^{q}(\mathbb {R} ^{d})}^{1-\theta }
\end{equation*}
\end{proposition}

The classical Gagliardo Nirenberg inequality
\begin{equation*}
    \|D^{j}f\|_{L^{p}(\mathbb {R} ^{d})}\leq C_G\|D^{m}f\|_{L^{r}(\mathbb {R} ^{d})}^{\theta }\|f\|_{L^{q}(\mathbb {R} ^{d})}^{1-\theta }
\end{equation*}
can be found in \cite{nirenberg1959elliptic}.
To generalize the inequality we can use the interpolation
\begin{equation*}
    \|D^{j}f\|_{L^{p,\alpha}}\le C_I \|D^{j}f\|_{L^{p_1}}^{\hat\theta}\|D^{j}f\|_{L^{p_2}}^{1-\hat\theta}.
\end{equation*}
with 
\begin{equation*}
    p_1< p <p_2\quad\text{and}\quad \frac{1}{p}=\frac{\hat\theta}{p_1}+\frac{1-\hat \theta}{p_2},
\end{equation*} $p,p_1,p_2$ close enough. Then we can apply Gagliardo Nirenberg for both $\|D^{j}f\|_{L^{p_1}}, \|D^{j}f\|_{L^{p_2}}$ w.r.t. $\|D^{m}f\|_{L^{r}},\|f\|_{L^{q}}$: by denoting $\theta_1,\theta_2$ as exponents such that
\begin{equation*}
    \frac {1}{p_1}={\dfrac {j}{d}}+\theta_1 \left({\dfrac {1}{r}}-{\dfrac {m}{d}}\right)+{\dfrac {1-\theta_1 }{q}},
\end{equation*}
\begin{equation*}
    \frac {1}{p_2}={\dfrac {j}{d}}+\theta_2 \left({\dfrac {1}{r}}-{\dfrac {m}{d}}\right)+{\dfrac {1-\theta_2 }{q}},
\end{equation*}
by classical Gagliardo Nirenberg
\begin{equation*}
    \begin{split}
            \|D^{j}f\|_{L^{p,\alpha}}&\le C_I \|D^{j}f\|_{L^{p_1}}^{\hat\theta}\|D^{j}f\|_{L^{p_2}}^{1-\hat\theta}\\
            &\le C_{GL}\|D^{m}f\|_{L^{r}(\mathbb {R} ^{d})}^{\theta_1\hat \theta+\theta_2(1-\hat \theta)}\|f\|_{L^{q}(\mathbb {R} ^{d})}^{(1-\theta_1)\hat\theta+(1-\theta_2)(1-\hat\theta) },
    \end{split}
\end{equation*}
with $C_{GL}=C_{I}C_{G}C_{G}$. It is easy to prove that
\begin{equation*}
    \theta_1\hat \theta+\theta_2(1-\hat \theta)=\theta,\quad (1-\theta_1)\hat\theta+(1-\theta_2)(1-\hat\theta)=1-\theta.
\end{equation*}

\subsection{The heat kernel} Denote
\begin{equation*}
  G(t,x)=\frac1{(4\pi t)^{d/2}}
  \exp\left(-\frac{|x|^2}{4t}\right)
\end{equation*}
the heat kernel. It is easy to prove with a direct computation or by interpolation the following lemma.
\begin{lemma}\label{lem:heat}
For every $r\in[1,\infty)$ there is $C_{HK}=C_{HK}(d,r)$ with
\begin{equation}\label{eq:heatkernel}
\norm{\nabla G_t}_{r,1}\le C_{HK}\,t^{-\frac12-\frac d2\left(1-\frac1r\right)},\qquad t>0 .
\end{equation}
\end{lemma}

\section{Local existence and continuation}\label{sec:local}
The purpose of this section is to demonstrate the well-posedness and non-negativity of the solution, as well as to justify certain regularity assumptions that will prove useful later in the article.
\subsection{Local well-posedness}
We prove the local well-posedness of the solution.

\begin{proposition}[Local bounded solutions]\label{prop:local}
There is a maximal time $T_*\in(0,\infty]$ and a unique solution
\begin{equation*}
  u\in C([0,T_*);L^1(\R^d))
  \cap L^\infty_{\mathrm{loc}}
  ([0,T_*)\times\R^d)
\end{equation*}
satisfying
\begin{equation}\label{eq:duhamel}
  u(t)=G(t)*u_0-
  \int_0^t\nabla G(t-s)*\bigl(b(s)u(s)^{1+k}\bigr)\,d s.
\end{equation}
The solution conserves the mass.  
\end{proposition}

\begin{proof}
The proof holds by showing that the Duhamel's formula \eqref{eq:duhamel} describes the fixed point of the contraction operator
$$\Phi[u](t)=G(t)*u_0-
  \int_0^t\nabla G(t-s)*\bigl(b(s)u(s)^{1+k}\bigr)\,d s
$$ in norm $L^\infty\cap L^1$ on the set 
$$
S = \left\{ u \in {C([0,T], L^1(\R^d))\cap L^\infty([0,T],L^\infty(\R^d))}: \|u\|_\infty \le R,\, u(0) = u_0 \right\}.
$$
First, for $u,v\in S$, we compute 
\begin{equation*}
    \norm{b u^{k+1}-b v^{k+1}}_{p,\infty}\le B_T (1+k)R^k \norm{u-v}_\infty,
\end{equation*}
and by Holder
\begin{equation}\label{eq:loc2}
\begin{split}
\norm{b\bigl(u^{k+1}-v^{k+1}\bigr)}_{1}
&\le B_T\norm{u^{k+1}-v^{k+1}}_{p/(p-1),1}\\
&\le (1+k)R^k B_T \norm{u-v}_1^{1-\frac1p}\norm{u-v}_\infty^{\frac1p}\\
&\le (1+k)R^k B_T\bigl(\norm{u-v}_1+\norm{u-v}_\infty\bigr).
\end{split}
\end{equation}
Moreover
\begin{equation*}
\norm{\nabla G_{t-s}\star u}_\infty\le\norm{\nabla G_{t-s}}_{p/(p-1),1}\norm{u}_{p,\infty}
\le C_{HK}(t-s)^{-\frac12-\frac d{2p}}\norm{u}_{p,\infty},
\end{equation*}
and
\begin{equation*}
\norm{\nabla G_{t-s}\star u}_1\le\norm{\nabla G_{t-s}}_{1}\norm{u}_1\le C_{HK}(t-s)^{-\frac12}\norm{u}_1 .
\end{equation*}
Using these estimates, it is easy to prove the following facts, that conclude the proof:
\begin{enumerate}
\item $\lim_{\varepsilon\to 0^+} \|\Phi[u](t+\varepsilon) - \Phi[u](t)\|_\infty =0,$
\item for every $t_0>0$, $\lim_{\varepsilon\to 0^+} \|\Phi[u](\cdot+\varepsilon) - \Phi[u](\cdot)\|_\infty =0$ uniformly on $[t_0,T],$
\item  $\lim_{\varepsilon\to 0^+} \|\Phi[u](\cdot+\varepsilon) - \Phi[u](\cdot)\|_1 =0$ uniformly on $[0,T]$,
\item $\Phi[u](t)$ is a contraction in norm $L^1$ and in norm $L^\infty$,
\item for $T\ll1$ it holds $\Phi(S)\subseteq S$
\item solution preserves the mass
\end{enumerate}
See also \cite{leccese2026existence} for similar computations and more details. This completes the proof. 
\end{proof}

\subsection{Energy regularity estimate}
We prove a standard energy regularity.
\begin{proposition}\label{prop:energy_reg}
  It holds for $T<T^*$
\begin{equation*}
    u\in L^\infty((0,T),L^2(\mathbb R^d))
    \cap L^2((0,T),H^1(\mathbb R^d)),
    \qquad
    \partial_tu\in L^2((0,T),H^{-1}(\mathbb R^d)).
\end{equation*}
\end{proposition}
\begin{proof}
On every interval $[0,T]$ on which the solution is well defined in $L^1\cap L^\infty$, the function satisfies
\begin{equation}\label{eq:flux}
bu^{1+k}\in L^\infty\bigl([0,T],L^1\cap L^2\bigr).
\end{equation}
Indeed it holds
\begin{equation*}
\begin{split}   
\norm{bu^{1+k}}_2&\le \norm{b}_{p,\infty}\norm{u^{1+k}}_{p',2}\\
&= \norm{b}_{p,\infty}\norm{u}_{(1+k)p',\,2(1+k)}^{1+k}\\
&\le  C_I\norm{b}_{p,\infty}\|u\|_1^{1/p'}\|u\|_\infty^{1+k-1/p'},\\
&\le  C_I\norm{b}_{p,\infty}\bigl(\norm{u}_1+\norm{u}_\infty\bigr)^{1+k},
\end{split}
\end{equation*}
and similarly $\norm{bu^{1+k}}_1\le \norm{b}_{p,\infty}\norm{u^{1+k}}_{p/(p-1),1}<\infty$.
 
For an equation of the form $\partial_tz=\Delta z-\text{div } I$ with $z(0)\in L^2$ and $I=bu^{1+k}\in L^2\bigl((0,S)\times\R^d\bigr)$, a standard energy argument for the heat equation gives
\begin{equation*}
z\in L^\infty\bigl([0,S],L^2\bigr)\cap L^2\bigl(0,S;H^1\bigr),\qquad \partial_tz\in L^2\bigl([0,S],H^{-1}\bigr),
\end{equation*}
together with the identity
\begin{equation}\label{eq:energyid}
\frac12\norm{z(t)}_2^2+\int_0^t\norm{\nabla z(s)}_2^2\ ds
=\frac12\norm{z(0)}_2^2+\int_0^t\ \int I\cdot\nabla z\ dx\ ds .
\end{equation}
The claim follows by taking $z=u$.
\end{proof}
\subsection{Positiveness}
Now we prove that the solution is always non-negative.
\begin{proposition}
  The solution defined in Proposition \ref{prop:local} is non-negative
\end{proposition}
\begin{proof}
Using the Duhamel local fixed point with $F(z)=(z_+)^{1+k}$, the function $bu^{1+k}$ vanishes on $\{u<0\}$, so testing the equation with $-u_-$ and using \eqref{eq:energyid} gives
\begin{equation*}
\frac12\norm{u_-(t)}_2^2+\int_0^t\norm{\nabla u_-(s)}_2^2\ ds\le0,
\end{equation*}
whence $u\ge0$ and $F(u)=u^{1+k}$.
\end{proof}

\subsection{Comparison and conservation of the increment mass}\label{ssec:comparison}
 
\begin{lemma}\label{lem:comparison}
Let $u,v$ be solutions of \eqref{eq:pde}, and assume $u_0\ge v_0\ge0$. Then $u\ge v$ and
\begin{equation}\label{eq:consdelta}
\norm{u(t)-v(t)}_1=\int\bigl(u_0-v_0\bigr)\ dx .
\end{equation}
\end{lemma}
 
\begin{proof}
Let $w=u-v$ and
\begin{equation*}
h(t,x)=
\begin{cases}
\dfrac{u^{1+k}-v^{1+k}}{u-v}, & u\neq v,\\
(1+k)u^{k}, & u=v .
\end{cases}
\end{equation*}
Then $\partial_tw=\Delta w-\dive(hb w)$, and by the mean value theorem
\begin{equation}\label{eq:betabound}
\norm{h(t)b(t)}_{p,\infty}\le(1+k)B_T\max\bigl\{\norm{u(t)}_\infty,\norm{v(t)}_\infty\bigr\}^{k}=C_{h},
\end{equation}
on $t\in [0,T]$. Testing with $-w_-$, it holds
\begin{equation*}
\frac12\frac{\mathrm d}{\ dt}\norm{w_-}_2^2+\norm{\nabla w_-}_2^2
=-\int h(t)b(t)\,w_-\cdot\nabla w_-\ dx
\le C\norm{h(t)b(t)}_{p,\infty}\norm{w_-}_{p',2}\norm{\nabla w_-}_2 ,
\end{equation*}
By Young's inequality with the exponents $\frac{2}{1+d/p}$ and $\frac{2}{1-d/p}$,
\begin{equation*}
\norm{h(t)b(t)}_{p,\infty}\norm{\nabla w_-}_2^{1+\frac dp}\norm{w_-}_2^{1-\frac dp}
\le\norm{\nabla w_-}_2^2+C\norm{h(t)b(t)}_{p,\infty}^{\frac{2p}{p-d}}\norm{w_-}_2^2 ,
\end{equation*}
so that
\begin{equation*}
\frac{\mathrm d}{\ dt}\norm{w_-}_2^2\le CC_h^{\frac{2p}{p-d}}\norm{w_-}_2^2 .
\end{equation*}
Gronwall's inequality gives $w_-\equiv0$ on $[0,T]$, i.e. $u\ge v$. Identity \eqref{eq:consdelta} then follows from the non-negativity of $w$ and from the conservation of mass for $u$ and for $v$.
\end{proof}

\section{Global existence}\label{sec:globex}
In Subsection \ref{ssec:l2linf} we generalize the interpolation argument of \cite{guidolin2022global} (\cite{bianchini2024existence} in Lorentz spaces) to the multidimensional case. In particular this argument provides the global existence in the subcritical cases, and reduces the critical case from an $L^\infty$ estimate to an $L^2$ estimate. 
Subsection \ref{ssec:increment} contains the main result of the paper, where the $L^2$ norm of every critical solution is estimated by comparison with ordered solutions.
\subsection{Subcritical case and the passage from $L^2$ to $L^\infty$}\label{ssec:l2linf}
\begin{theorem}\label{thm:Guid}
Assume \eqref{eq:assumptions}. For $k<k_c 2^N$ and $M>N$, it holds
\begin{equation*}
\begin{split}
\max_{t \in [0,T]} \|u(t)\|_{2^M} \leq \Big( C(d,k,p)\Big)^{\bar \alpha} 2^{\bar \beta} \Big(B_T\Big)^{\bar \zeta}\max \bigg\{ &\|u(0)\|_{2^M}, 
\max_{t \in [0,T]} \|u\|_{2^{N}}^{\gamma(N,M)} \bigg\}.
\end{split}
\end{equation*}
\end{theorem}
\begin{proof}
From Equation \eqref{eq:pde} it holds 
\begin{equation}\label{eq:first_guidolin}
\begin{split}
\frac{d}{dt}\int u^{2n}\,dx
&=-\frac{2(2n-1)}{n}\norm{\nabla(u^n)}_2^2
+2(2n-1)\int \nabla(u^n)\cdot b\,u^{n+k}\,dx\\
&\le-\frac{2(2n-1)}{n}\norm{\nabla(u^n)}_2^2
+2(2n-1)\norm{b\,u^{n+k}}_2\norm{\nabla(u^n)}_2\\
&\le-\frac{2(2n-1)}{n}\norm{\nabla(u^n)}_2^2
+\,2(2n-1)\,B_T\,\norm{u^n}_{(1+\frac kn)p',\,2(1+\frac kn)}^{(1+\frac kn)}\norm{\nabla(u^n)}_2\\
&\le-\frac{(2n-1)}{n}\norm{\nabla(u^n)}_2^2
+\,n(2n-1)\,B_T^2\,\norm{u^n}_{(1+\frac kn)p',\,2(1+\frac kn)}^{2(1+\frac kn)}.
\end{split}
\end{equation}

In particular, if $\norm{u}_{2n}$ is increasing, then
\begin{equation}\label{eq:incr}
\norm{\nabla(u^n)}_2\le n\,B_T\,\norm{u^n}_{(1+\frac kn)p',\,2(1+\frac kn)}^{(1+\frac kn)}.
\end{equation}
Gagliardo--Nirenberg in Lorentz spaces, with the exponent $\theta$ given by
\begin{equation}\label{eq:GNun}
\frac{1}{(1+\frac kn)p'}=\Bigl(\frac12-\frac1d\Bigr)\theta+1-\theta.
\end{equation}
gives
\begin{equation*}
\norm{u^n}_{(1+\frac kn)p',\,2(1+\frac kn)}^{(1+\frac kn)}
\le C(d,p,k)\,\norm{\nabla(u^n)}_2^{(1+\frac kn)\theta}\norm{u^n}_1^{(1+\frac kn)(1-\theta)} ,
\end{equation*}
From \eqref{eq:GNun} and \eqref{eq:incr}, it holds
\begin{equation*}
\norm{\nabla(u^n)}_2\le\bigl(C(d,p,k)\,nB_T\bigr)^{\frac{1}{1-(1+\frac kn)\theta}}
\norm{u^n}_1^{\frac{(1+\frac kn)(1-\theta)}{1-(1+\frac kn)\theta}} .
\end{equation*}
Using again Gagliardo--Nirenberg, if $\norm{u}_{2n}$ is increasing then
\begin{equation*}
\norm{u^n}_2\le C(d)\norm{\nabla(u^n)}_2^{\frac{d}{d+2}}\norm{u^n}_1^{\frac{2}{d+2}}
\le C(d,p,k)\,\bigl(nB_T\bigr)^{\frac{\frac12}{\frac1d-\frac kn-\frac1p}}
\norm{u^n}_1^{\frac{1-\frac dp-\frac{dk}{2n}}{1-\frac dp-\frac{dk}{n}}},
\end{equation*}
or equivalently
\begin{equation*}
\norm{u}_{2n}\le\bigl(C(d,k,p)\bigr)^{\frac1n}\bigl(nB_T\bigr)^{\frac{\frac1{2n}}{\frac1d-\frac kn-\frac1p}}
\norm{u}_{n}^{\frac{1-\frac dp-\frac{dk}{2n}}{1-\frac dp-\frac{dk}{n}}} .
\end{equation*}
It is now easy to deduce that
\begin{equation*}
\max_{t\in[0,T]}\norm{u(t)}_{2n}\le\max\Bigl\{\norm{u(0)}_{2n},\
\bigl(C(d,k,p)\bigr)^{\frac1n}\bigl(nB_T\bigr)^{\frac{\frac1{2n}}{\frac1d-\frac kn-\frac1p}}
\max_{t\in[0,T]}\norm{u(t)}_{n}^{\frac{1-\frac dp-\frac{dk}{2n}}{1-\frac dp-\frac{dk}{n}}}\Bigr\}.
\end{equation*}
Iterating for $n=2^l$, $l=N,\dots,M-1$, one obtains
\begin{equation*}
\max_{t\in[0,T]}\norm{u(t)}_{2^M}\le\max\Bigl\{\norm{u(0)}_{2^M},\
\bigl(C(d,k,p)\bigr)^{\alpha(l,M)}2^{\beta(l,M)}B_T^{\zeta(l,M)}
\max_{t\in[0,T]}\norm{u(t)}_{2^{l}}^{\gamma(l,M)}\Bigr\}_{l=N,\dots,M-1},
\end{equation*}
where
\begin{equation*}
\gamma(N,M)=\prod_{n=N+1}^{M}\frac{1-\frac dp-\frac{dk}{2^{n}}}{1-\frac dp-\frac{dk}{2^{n-1}}}
=\frac{1-\frac dp-\frac{dk}{2^{M}}}{1-\frac dp-\frac{dk}{2^{N}}},
\qquad\gamma(M,M)=1,
\end{equation*}
\begin{equation*}
\alpha(N,M)=\sum_{n=N}^{M-1}2^{-n}\gamma(n+1,M)\le\bar\alpha,
\qquad
\beta(N,M)=\sum_{n=N}^{M-1}\frac{n\,2^{-n-1}}{\frac1d-\frac{k}{2^{n}}-\frac1p}\gamma(n+1,M)\le\bar\beta,
\end{equation*}
\begin{equation*}
\zeta(N,M)=\sum_{n=N}^{M-1}\frac{2^{-n-1}}{\frac1d-\frac{k}{2^{n}}-\frac1p}\gamma(n+1,M)\le\bar\zeta,
\end{equation*}
the three series converging, so that
$\bar\alpha,\bar\beta,\bar\zeta$ depend only on $d,k,p$ and $N$. By repeatedly applying
interpolation and Young's inequality,
\begin{equation*}
\begin{split}
\norm{u(0)}_{2^{l}}^{\gamma(l,M)}
&\le\bigl(\norm{u(0)}_{2^{l+1}}^{\frac23}\norm{u(0)}_{2^{l-1}}^{\frac13}\bigr)^{\gamma(l,M)}\\
&\le\frac{2\gamma(l,M)}{3\gamma(l+1,M)}\norm{u(0)}_{2^{l+1}}^{\gamma(l+1,M)}
+\frac{\gamma(l,M)}{3\gamma(l-1,M)}\norm{u(0)}_{2^{l-1}}^{\gamma(l-1,M)}\\
&\le\max\bigl\{\norm{u(0)}_{2^{l+1}}^{\gamma(l+1,M)},\ \norm{u(0)}_{2^{l-1}}^{\gamma(l-1,M)}\bigr\},
\end{split}
\end{equation*}
the intermediate terms are interpolated into the two extreme ones, and the statement follows.
\end{proof}

Taking  $M\to\infty$ and respectively $N=0$ or $N=1$ the two corollaries follow.
\begin{corollary}\label{cor:subc}
  Assume \eqref{eq:assumptions} and $k<k_c$. It holds
\begin{equation*}
\norm{u}_{\infty}\le C(d,k,p)\,B_T^{\bar\zeta}\max\bigl\{\norm{u(0)}_{\infty},\
\norm{u(0)}_1^{\frac{1-\frac dp}{1-\frac dp-dk}}\bigl\}.
\end{equation*}
\end{corollary}

\begin{corollary}\label{cor:critL2Linfty}
   Assume \eqref{eq:assumptions} and $k=k_c$. It holds
\begin{equation*}
\norm{u}_{\infty}\le C(d,k,p)\,B_T^{\bar\zeta}\max\bigl\{\norm{u(0)}_{\infty},\
\max_{t\in[0,T]}\norm{u(t)}_{2}^{2}\bigr\}.
\end{equation*}
\end{corollary}

\subsection{Energy estimate for the increment}\label{ssec:increment}
Let $u\ge v$ be two solutions of \eqref{eq:pde}, and set
\begin{equation*}
w:=u-v\ge0,\qquad \delta:=\norm{w(0)}_1\overset{\eqref{eq:consdelta}}{=}\norm{w(t)}_1.
\end{equation*}
Subtracting the two equations, it holds
\begin{equation}\label{eq:weq}
\partial_tw+\dive\bigl(b\,H(v,w)\bigr)=\Delta w,\qquad H(v,w):=(v+w)^{1+k}-v^{1+k},
\end{equation}
in the sense of distributions. Since $0<k\le1$, the map $s\mapsto s^{k}$ is subadditive, whence
\begin{equation}\label{eq:split}
\begin{split}
H(v,w)&=(1+k)\int_0^w(v+s)^{k}\ ds\\
&\le(1+k)\int_0^w\bigl(v^{k}+s^{k}\bigr)\ ds\\
&=(1+k)v^{k}w+w^{1+k}\\
&\le(1+k)\norm{v(t)}_\infty^k\,w+w^{1+k}.
\end{split}
\end{equation}
From \eqref{eq:weq} it holds
\begin{equation}\label{eq:w}
  \begin{split}
\frac d{dt}\norm{w}_2^2+\norm{\nabla w}_2^2&\le\int|b|^2H(v,w)^2\ dx\\
&\le 2\int|b|^2w^{2+2k}\ dx+2(1+k)^2\norm{v(t)}_\infty^{2k}\int|b|^2w^2\ dx
\end{split}
\end{equation}
We treat the last two terms separately.
\begin{equation*}
  \begin{split}
2\int|b|^2w^{2+2k}\ dx &\le  2\norm{b}_{p,\infty}^2\norm{w}_{2(k+1)p/(p-2),2(k+1)}^{2(k+1)}\\
&\le  2 C_{GL}\norm{b}_{p,\infty}^2\delta^{2k}\norm{\nabla w}_2^2
\end{split}
\end{equation*}
and defining
\begin{equation}\label{eq:deltastar}
\delta_*:=\begin{cases}
\bigl(8\,C_{GL}\sup_{t
\in[0,T]}\norm{b}_{p,\infty}^2\bigr)^{-1/(2k)}&\sup_{t
\in[0,T]}\norm{b}_{p,\infty}\ne 0\\
+\infty &\sup_{t
\in[0,T]}\norm{b}_{p,\infty}=0
\end{cases}
\end{equation}
so that
\begin{equation}\label{eq:absorbed}
\delta\le\delta_*\qquad\Longrightarrow\qquad 2\int|b|^2w^{2(k+1)}\ dx\le\frac14 \norm{\nabla w}_2^2 .
\end{equation}
 For the other term, 
\begin{equation}\label{eq:backterm}
\begin{split}
2(1+k)^2\norm{v(t)}_\infty^{2k}\int|b|^2w^2\ dx
&\le 2(1+k)^2C_I\,\norm{v(t)}_\infty^{2k}\norm{b}_{p,\infty}^2\norm{\nabla w}_2^{\frac{2d}{p}}\norm{w}_2^{2\left(1-\frac dp\right)}\\
&\le\frac14 \norm{\nabla w}_2^2+C_{GL}C_I\norm{b}_{p,\infty}^{\frac{2p}{p-d}}\norm{v(t)}_\infty^{\frac{2kp}{p-d}}\norm{w}_2^2\\
&=\frac14 \norm{\nabla w}_2^2+C_{GL}C_I\norm{b}_{p,\infty}^{\frac{2p}{p-d}}\norm{v(t)}_\infty^{2/d}\norm{w}_2^2 .
\end{split}
\end{equation}
Combining \eqref{eq:w}, \eqref{eq:absorbed}, \eqref{eq:backterm} it follows the next theorem.

\begin{theorem}[Critical $L^2$ estimate for ordered increments]\label{thm:increment}
Let $d\ge2$, $d<p\le\infty$, $k=k_c$, and let $u\ge v$ be solutions of \eqref{eq:pde} with $u(0)\ge v(0)\ge0$ and $\delta\le \delta_*$. Then for every $0\le t\le T< T^*$
\begin{equation}\label{eq:incrementode}
\frac{\mathrm d}{\ dt}\norm{w(t)}_2^2+\frac12\norm{\nabla w(t)}_2^2
\le C_{GL}C_I\, B_T^{\frac{2p}{p-d}}\,\norm{v(t)}_\infty^{2/d}\,\norm{w(t)}_2^2 ,
\end{equation}
and consequently, for every $t< T^*,$
\begin{equation}\label{eq:incrementbound}
\norm{w(t)}_2^2\le\norm{w(0)}_2^2\,
\exp\Bigl(C_{GL}C_I{B_{T^*}}^{\frac{2p}{p-d}}\int_0^t\norm{v(r)}_\infty^{2/d}\ dr\Bigr).
\end{equation}
\end{theorem}
 
\subsection{Arbitrary mass in finitely many steps}\label{sec:chain}

Theorem \ref{thm:increment} is applied along a finite chain of data interpolating between $0$ and the datum of interest. 
We state the conclusion in the conditional form in which it is used: the chain propagates $L^2$ bounds, and it needs, at each step, an $L^\infty$ bound for the solution constructed at the previous step. 
Producing that bound from an $L^2$ bound is the content of Corollary \ref{cor:critL2Linfty}.
 
We define
\begin{equation}\label{eq:L2estimate}
\begin{split}
&L_i=\frac{\sqrt{\norm{u_0}_1\norm{u_0}_\infty}}{N}\sum_{l=0}^{i-1}e^{C_L B_T^{\frac{2p}{p-d}} F_l^{2/d}},\\
&F_l = C_*\max\bigl\{\norm{u_0}_\infty,\,L_l^2\bigr\},\\
&C_*=C(d,k_c,p)B_T^{\bar\zeta}\ \text{as in Corollary 4.3,}\\
&C_L=\frac{C_{GL}C_I T}{2},
\end{split}
\end{equation}
so that 
\begin{equation*}L_{j-1}+\frac{\sqrt{\norm{u_0}_1 \norm{u_0}_\infty}}{N}\,
\exp\Bigl(C_L B_T^{\frac{2p}{p-d}}\,F_{j-1}^{2/d}\Bigr)=:L_j
\end{equation*}
\begin{theorem}[The chain]\label{thm:chain}
Let $d\ge2$, $d<p\le\infty$, and define
\begin{equation}\label{eq:N}
N:=\max\left\{\left\lceil\frac {\norm{u_0}_1}{\delta_*}\right\rceil,1\right\}.
\end{equation}
For $j=0,1,\dots,N$ let $u_{(j)}$ be the solution of \eqref{eq:pde} with initial datum
\begin{equation*}
u_{(j)}(0,\cdot)=\frac jN\,u_0 ,
\end{equation*}
so that $u_{(0)}\equiv0$. Assume that for some $j\ge1$ the previous solution obeys
\begin{equation*}
\sup_{t\le T}\bigl\lVert u_{(j-1)}(t)\bigr\rVert_\infty\le F_{j-1},
\qquad
\sup_{t\le T}\bigl\lVert u_{(j-1)}(t)\bigr\rVert_2\le L_{j-1} .
\end{equation*}
Then $u_{(j)}\ge u_{(j-1)}$, the increment has the conserved mass $\norm{u_0}_1/N\le\delta_*$, and
\begin{equation}\label{eq:chainbound}
\sup_{t\le T}\bigl\lVert u_{(j)}(t)\bigr\rVert_2\le L_{j-1}+\frac{\sqrt{\norm{u_0}_1 \norm{u_0}_\infty}}{N}\,
\exp\Bigl(C_L B_T^{\frac{2p}{p-d}}\,F_{j-1}^{2/d}\Bigr)=:L_j .
\end{equation}
In particular, for $u_{(N)}=u$ it holds a $L^2$ bound for the solution with datum $u_0$, and the constant $L_N$ depends only on $d,p,T,B_T,\norm{u_0}_1$ and $\norm{u_0}_\infty$.
\end{theorem}
 
\begin{proof}
Since $\frac jNu_0\ge\frac{j-1}Nu_0\ge0$, Lemma \ref{lem:comparison} gives $u_{(j)}\ge u_{(j-1)}$ and
\begin{equation*}
\bigl\lVert u_{(j)}(t)-u_{(j-1)}(t)\bigr\rVert_1=\frac1N\norm{u_0}_1\le\delta_*.
\end{equation*}
Theorem \ref{thm:increment}, applied to the pair $(u_{(j)},u_{(j-1)})$, therefore gives
\begin{equation*}
\begin{split}
\bigl\lVert u_{(j)}(t)-u_{(j-1)}(t)\bigr\rVert_2
&\le\frac{\norm{u_0}_2}N\exp\Bigl(\frac{C_{GL}C_I}{2}B_T^{\frac{2p}{p-d}}\int_0^t\bigl\lVert u_{(j-1)}(s)\bigr\rVert_\infty^{2/d}\ ds\Bigr)\\
&\le\frac{\sqrt{\norm{u_0}_1 \norm{u_0}_\infty}}N\exp\Bigl(C_L B_T^{\frac{2p}{p-d}}F_{j-1}^{2/d}\Bigr).
\end{split}
\end{equation*}
The triangle inequality
$\norm{u_{(j)}}_2\le\norm{u_{(j-1)}}_2+\norm{u_{(j)}-u_{(j-1)}}_2$
gives \eqref{eq:chainbound}.
\end{proof}

\begin{proof}[Proof of Theorem \ref{thm:global-main}]
In the subcritical case the bound is Corollary \ref{cor:subc}. 
Let $k=k_c$ and let $N$ be as in \eqref{eq:N}. We argue by induction on $j$. $u_{(0)}=0$ satisfies the two bounds with $L_0=0$. 
If $u_{(j-1)}$ obeys 
\begin{equation*}
    \sup_{t\le T}\norm{u_{(j-1)}}_2\le L_{j-1}
\end{equation*} and
\begin{equation*}
    \sup_{t\le T}\norm{u_{(j-1)}}_\infty\le F_{j-1},
\end{equation*}
then Theorem \ref{thm:chain} gives the $L^2$ bound $L_j$ for $u_{(j)}$, 
and Corollary \ref{cor:critL2Linfty} converts it into the $L^\infty$ bound $F_j$. 
After $N$ steps $u_{(N)}=u$ has datum $u_0$, and $F_N$ depends only on $d,p,k,T,B_T,\norm{u_0}_1,\norm{u_0}_\infty$.
\end{proof}

\section{Long-time behaviour of solutions}
\label{sec:longbehav}

In this last section we assume \eqref{eq:assumptions2} and we set $B_\infty:=\sup_{t>0}\norm{b(t)}_{p,\infty}$.

\begin{proposition}\label{prop:stationary}
Let $0<k<k_c$ and let $\frac d2<\alpha<\frac{1-d/p}{2k}$, an interval which is non-empty
precisely because $k<k_c$. Then
\begin{equation*}
u(x)=\bigl(1+|x|^2\bigr)^{-\alpha},\qquad b(x)=-2\alpha\,x\,\bigl(1+|x|^2\bigr)^{\alpha k-1}
\end{equation*}
satisfy $u\in L^1\cap L^\infty$, $b\in L^{p}(\R^d;\R^d)$, and $u$ is a stationary solution
of \eqref{eq:pde}. In particular no decay holds in general in the subcritical regime.
\end{proposition}

\begin{proof}
Since $2\alpha>d$ we have $u\in L^1\cap L^\infty$, and
$b\,u^{1+k}=-2\alpha x(1+|x|^2)^{-\alpha-1}=\nabla u$, so $\nabla u-bu^{1+k}$
vanishes identically. Finally $|b(x)|\le C(1+|x|)^{2\alpha k-1}$ with
$2\alpha k-1<-\frac dp$, whence $b\in L^p$.
\end{proof}

\begin{proof}[Proof of Theorem \ref{Theo:long_time_1}]
{\it Point \eqref{Point_1:long_time_1}.} This is exactly Proposition \ref{prop:stationary}.

{\it Point \eqref{Point_2:long_time_1}.} Let $N$ be as in \eqref{eq:N}, with $\delta_*$ computed with $B_\infty$. Let $u_{(j)}$ be the solutions of Theorem \ref{thm:chain} and set
\begin{equation*}
    t=e^\tau-1,\qquad x=e^{\tau/2}y,\qquad v_{(j)}(\tau,y)=e^{d\tau/2}u_{(j)}(t,x).
\end{equation*}
Each $v_{(j)}$ solves
\begin{equation*}
    \partial_\tau v=\Delta_yv+\frac12\dive_y(yv)-\dive_y(\tilde b\,v^{1+k})
\end{equation*}
with $\norm{\tilde b(\tau)}_{p,\infty}=\norm{b(e^\tau-1)}_{p,\infty}\le B_\infty$, and
$w_j=v_{(j)}-v_{(j-1)}\ge0$ has mass $\norm{u_0}_1/N\le\delta_*$. Arguing as in Subsection
\ref{ssec:increment}, the additional term gives $\frac d2\norm{w_j}_2^2$, and
Gagliardo--Nirenberg
\begin{equation*}
    \norm{w_j}_2^{2+\frac4d}\le C_G\norm{\nabla w_j}_2^2(\norm{u_0}_1/N)^{\frac4d}
\end{equation*}
gives
\begin{equation*}
    \frac{d}{d\tau}\norm{w_j}_2^2\le A_j\norm{w_j}_2^2-C_W\,\norm{w_j}_2^{2+\frac4d},
    \qquad
    A_j=\frac d2+C B_\infty^{\frac{2p}{p-d}}\sup_{\tau>0}\norm{v_{(j-1)}(\tau)}_\infty^{2/d},
\end{equation*}
with $C_W=\frac1{2C_G}\bigl(N/\norm{u_0}_1\bigr)^{\frac4d}$. Since the right-hand side is negative for $\norm{w_j}_2^2>(A_j/C_W)^{d/2}$,
it follows that $\sup_\tau\norm{v_{(j)}}_2<\infty$ whenever $\sup_\tau\norm{v_{(j-1)}}_\infty<\infty$.
We now apply the iteration of Theorem \ref{thm:Guid} to $v_{(j)}$. The only additional term gives
\begin{equation*}
    \frac{(2n-1)d}2\norm{v^n}_2^2\le\frac{2n-1}{2n}\norm{\nabla(v^n)}_2^2+Cn^{1+\frac d2}\norm{v^n}_1^2 .
\end{equation*}
Since $v_{(j)}(0)=\frac jNu_0$ is bounded, Corollary \ref{cor:critL2Linfty} then gives
$\sup_\tau\norm{v_{(j)}}_\infty<\infty$.

Starting from $v_{(0)}\equiv0$, induction on $j$ yields $\sup_\tau\norm{v(\tau)}_\infty<\infty$
for $v=v_{(N)}$. Since 
\begin{equation*}
\norm{u(t)}_\infty=(1+t)^{-d/2}\norm{v(\tau)}_\infty,
\end{equation*} the claim
follows.

{\it Point \eqref{Point_3:long_time_1}.}  We can write
\begin{equation*}
b\,u^{1+k}=\hat b\,u^{1+k_c},\qquad \hat b:=b\,u^{\,k-k_c},
\end{equation*}
so that $u$ solves \eqref{eq:pde} with the critical exponent $k_c$ and the field $\hat b$,
which satisfies
\begin{equation*}
\norm{\hat b(t)}_{p,\infty}\le\norm{b(t)}_{p,\infty}\norm{u}_{L^\infty_{t,x}}^{\,k-k_c}
\le B_\infty\norm{u}_{L^\infty_{t,x}}^{\,k-k_c}<\infty ,
\end{equation*}
hence \eqref{eq:assumptions2}. Point \eqref{Point_2:long_time_1} applies to $\hat b$ and gives the claim.
\end{proof}
\printbibliography

\end{document}